\documentclass[12pt, twoside]{article}
\usepackage{amsmath,amsthm,amssymb}
\usepackage{times}
\usepackage{enumerate}
\usepackage{extarrows}
\usepackage{color}
\usepackage{tabularx}
\usepackage{multirow}
\usepackage {graphics}
\usepackage{graphicx}

\def\titlerunning#1{\gdef\titrun{#1}}
\makeatletter
\def\author#1{\gdef\autrun{\def\and{\unskip, }#1}\gdef\@author{#1}}
\def\address#1{{\def\and{\\\hspace*{18pt}}\renewcommand{\thefootnote}{}%
\footnote {#1}}%
\markboth{\autrun}{\titrun}}
\makeatother
\def\email#1{e-mail: #1}
\def\subjclass#1{{\renewcommand{\thefootnote}{}%
\footnote{\emph{Mathematics Subject Classification (2020):} #1}}}
\def\keywords#1{\par\medskip
\noindent\textbf{Keywords.} #1}

\newtheorem{theorem}{Theorem}[section]
\newtheorem{lemma}[theorem]{Lemma}
\newtheorem{definition}[theorem]{Definition}
\newtheorem{proposition}[theorem]{Proposition}
\newtheorem{remark}[theorem]{Remark}

\newtheorem{example}[theorem]{Example}
\newtheorem*{thmA}{\textbf{Theorem A}}
\newtheorem*{thmB}{\textbf{Theorem B}}

\newtheorem*{corC}{\textbf{Corollary C}}

\newcommand{\R}{\mathbb{R}}

\newcommand{\Proof}{\begin{proof}}
\newcommand{\End}{\end{proof}}

\numberwithin{equation}{section}

\newcommand{\PreserveBackslash}[1]{\let\temp=\\#1\let\\=\temp}
\newcolumntype{C}[1]{>{\PreserveBackslash\centering}p{#1}}
\newcolumntype{R}[1]{>{\PreserveBackslash\raggedleft}p{#1}}
\newcolumntype{L}[1]{>{\PreserveBackslash\raggedright}p{#1}}

\newcolumntype{I}{!{\vrule width 1pt}}
\newlength\savedwidth

\begin{document}

%%%%% To ease editing, add:

\baselineskip=15pt

%%%%%%%%%%%%%%%%

%% In the running head, give an abbreviation of the title.
\titlerunning{Weak KAM solutions of Hamilton-Jacobi equations}

\title{Viscosity subsoltions of Hamilton-Jacobi equations and Invariant sets of contact Hamilton systems}

\author{Xiang Shu\dag \and Jun Yan\ddag\and Kai Zhao*}

\date{\today}

\maketitle

\address{Xiang Shu: School of Mathematical Sciences, Fudan University, Shanghai 200433, China; \email{xshu19@fudan.edu.cn}
\and Jun Yan: School of Mathematical Sciences, Fudan University, Shanghai 200433, China;
\email{yanjun@fudan.edu.cn}
\and
Kai Zhao:  School of Mathematical Sciences, Fudan University, Shanghai 200433, China;
\email{zhao$\_$kai@fudan.edu.en}}

%%%%%%%%
\begin{abstract}
%We  give some descriptions of viscosity subsolutions of contact Hamilton-Jacobi equation on a connected, closed manifold $M$ 
%$$
%H(x,\partial_x u,u)= 0,
%$$
%by the positive invariant set under the contact phase flow generated by 
%\begin{equation}\nonumber
%\left\{
%\begin{aligned}
%\dot x&=\frac{\partial H }{\partial p}(x,p,u)\\
%\dot p &=-\frac{\partial H }{\partial x}(x,p,u)-\frac{\partial H }{\partial u}(x,p,u) \cdot p \quad (x,p,u)\in T^*M \times \R. \\ 
%\dot u&=\frac{\partial H }{\partial p}(x,p,u) \cdot p-H(x,p,u)
%\end{aligned}
%\right.
%\end{equation}
The objective of this paper is to present some results about viscosity subsolutions of the contact Hamiltonian-Jacobi equations on connected, closed manifold $M$ 
$$
H(x,\partial_x u,u)= 0, \quad x\in M.
$$
Based on implicit variational principles introduced in \cite{WWY2,WWY3}, we focus on the monotonicity of the solution semigroups on viscosity subsolutions and the positive invariance of the epigraph for viscosity subsolutions. Besides, we show a similar consequence for strict viscosity subsolutions on $M$.
%% Keywords are optional
\end{abstract}

%\newpage
%
%\tableofcontents

%\newpage
%%%%%%%%%%%%%%%%%%%%%%%%%%%%%%%%%%%%%%%%%%%%%%%%%%%%%%%%Sect. 1

\section{Introduction}
\setcounter{equation}{0}
\setcounter{footnote}{0}

Suppose $M$ is a closed (i.e.,compact, without boundary) connected and smooth manifold. Let $H(x,p,u)$: $T^*M \times \R\rightarrow \R$ be a $C^3$ function  satisfying:
\begin{itemize}
	\item[\textbf{(H1)}] Positive Definiteness:  $\frac{\partial^2 H}{\partial p^2}  (x,p,u) $ is positive definite on $T^*M\times \R$;
	\item[\textbf{(H2)}] Superlinearity: For every $(x,u)\in M\times \R $, $\lim_{|p|\to +\infty}\frac{H(x,p,u)}{|p|}=+\infty$;
	\item[\textbf{(H3)}] Lipschitz Continuity: There exists a constant $\lambda >0 $ such that $\Big| \frac{\partial H}{\partial u}(x,p,u)\Big|\leqslant \lambda $ 
	for any  $ (x,p,u)\in T^*M \times \R $; 
\end{itemize}

We consider the viscosity solutions of the  following first-order partial differential equation
	\begin{equation}\label{eq:intro_HJs}\tag{HJ$_s$}
	H(x,\partial _x u,u)=0, \quad x\in M. 
	\end{equation}

Let $J^1(M,\R)$ denote the manifold of 1-jet of functions on $M$. The standard contact form on $J^1(M,\R)$ is the 1-form $\alpha=d u-p d x $. Every $C^2$ function $H(x,u,p)$ determinates a unique vector field $X_H$ defined by the conditions
$$
\mathcal{L}_{X_H}\alpha=- \frac{\partial H}{\partial u} \alpha,\quad \alpha(X_H)=-H.
$$
where $\mathcal{L}_{X_H}$ denotes the Lie derivative along the contact vector field $X_H$. In Darboux coordinates, the contact vector field $X_H$ generated by $H$ is formulated by:
\begin{equation} \label{eq:ode}
X_H:\left\{
\begin{aligned}
\dot x&=\frac{\partial H }{\partial p}(x,p,u),\\
\dot p &=-\frac{\partial H }{\partial x}(x,p,u)-\frac{\partial H }{\partial u}(x,p,u) \cdot p, \quad (x,p,u)\in T^*M \times \R, \\ 
\dot u&=\frac{\partial H }{\partial p}(x,p,u) \cdot p-H(x,p,u).
\end{aligned}
\right.
\end{equation}
$H$ is called a contact Hamiltonian, and $X_H$ is called a contact Hamiltonian vector field. Due to the fundamental theorems for ordinary differential equations , for each $(x,p,u)\in T^*M\times \R$, there exists a unique integral curve of \eqref{eq:ode} through the point, which is denoted by $\Phi_H^t(x,p,u)$. In this paper, we need a extra condition on $\Phi_H^t$.
\begin{itemize}
	\item[\textbf{(SH1)}] Completeness: Every maximal integral curve of the contact Hamiltonian flow $\Phi_H^t$ has all of $\R$ as its domain of definition.
\end{itemize}
\begin{remark}
	In  Appendix A, we can give a common sufficient condition \eqref{eq:sufficient} to ensure \rm{(SH1)} holds.
\end{remark}
Actually, it is common to find non-trivial examples satisfying all of our assumptions. For instance, 
\begin{example}
	Consider the $C^3$-smooth Hamiltonian $H:T^*M \times \R \to \R$ defined by 
	$$H(x,p,u)=\frac{1}{2}\langle A(x)p,p \rangle +V(x,u),$$
	where $A(x)$ is  positive definite and $V(x,u):M\times \R \to \R$   satisfies $|\frac{\partial V}{\partial u}(x,u)|\leqslant \lambda $ for any $(x,u)\in M\times \R$.
%	 Since
%	\begin{align*}
%		\Big| p\cdot \frac{\partial H}{\partial p} (x,p,u)\Big|&=\langle A(x)p,p \rangle 
%		\leqslant 2(|H(x,p,u)|+|V(x,u)|) \\
%		&\leqslant 2(|H(x,p,u)|+\lambda |u|+|V(x,0)|) \\
%		&\leqslant A(H(x,p,u))(1+|u|),
%	\end{align*}
%	where $A(h):=2|h|+\max_{x\in M}|V(x,0)| + 2\lambda $.
%	Therefore, (H1)-(H3) and \eqref{eq:sufficient}  are satisfied.It shows  the completeness of the contact Hamiltonian flow $\Phi_H^t$ by Lemma \ref{lem:complete flow}.	
\end{example}

In this paper, we assume that $H(x,p,u)$ satisfies (H1)-(H3) and (SH1). Let $L(x,\dot{x},u)$ be the Legendre transform of $H(x,p,u)$,i.e.,
\begin{equation*}
	L(x,\dot{x},u)=\sup_{p \in T_x^*M}\{p \cdot \dot{x}-H(x,p,u)\}.
\end{equation*}

%\noindent{\bf 1.b.} -- {\it Solution semigroups}.
Let us recall two  semigroups of operators introduced in \cite{WWY1}.  Define a family of nonlinear operators $\{T^-_t\}_{t\geqslant0}$ from $C(M,\mathbb{R})$ to itself as follows. For each $\varphi\in C(M,\mathbb{R})$, we denote the unique continuous function on $ (x,t)\in M\times[0,+\infty)$ by $(x,t)\mapsto T^-_t\varphi(x)$ satisfying that
\begin{equation}\label{eq:T^-_t}
	T^-_t\varphi(x)=\inf_{\gamma}\left\{\varphi(\gamma(0))+\int_0^tL(\gamma(\tau),\dot{\gamma}(\tau),T^-_\tau\varphi(\gamma(\tau)))d\tau\right\},
\end{equation}
where the infimum is taken among the absolutely continuous curves $\gamma:[0,t]\to M$ with $\gamma(t)=x$.  It was also proved  in \cite{WWY1} that $\{T^-_t\}_{t\geqslant0}$ is a semigroup of operators 
and the function $(x,t)\mapsto T^-_t\varphi(x)$ is a viscosity solution of the evolutionary first-order partial differential equation \begin{equation}\label{eq:HJe}\tag{HJ$_e$}
\begin{cases}
\partial_t u+H(x,\partial_x u,u)=0,\quad (x,t)\in M\times(0,\infty),\\
u(x,0)=\varphi(x), \quad x\in M,
\end{cases}
\end{equation}
We call $\{T^-_t\}_{t\geqslant0}$ the \textit{backward solution semigroup}. 

Similarly, one can define another semigroup of operators $\{T^+_t\}_{t\geqslant0}$, called the \textit{forward solution  semigroup}, by
\begin{equation*}\label{fixufor}
T^+_t\varphi(x)=\sup_{\gamma}\left\{\varphi(\gamma(t))-\int_0^tL(\gamma(\tau),\dot{\gamma}(\tau),T^+_{t-\tau}\varphi(\gamma(\tau)))d\tau\right\},
\end{equation*}
where the supremum is taken among the absolutely continuous curves $\gamma:[0,t]\to M$ with $\gamma(0)=x$.

\medskip
It is well known that \eqref{eq:intro_HJs}  admits a viscosity solution through Perron's method following Ishii\cite{ishii1987}, and the fundamental idea of the proof is to find a special subset of viscosity subsolutions. Actually, viscosity subsolutions decide most of the properties of viscosity solutions, and we wish to clarify the properties of viscosity subsolutions. Recently, there are some dynamical results on the Aubry-Mather theory and the weak KAM theory for contact Hamiltonian systems \cite{WWY1,WWY2,WWY3,WWY4,Ni:2021aa},where variational principles \cite{WWY1,CCWY2019,CCJWY2020} played essential roles.

In this paper, our aim is to state some necessary and sufficient conditions for the viscosity subsolutions. To be specific, we can show that the viscosity subsolution is completely decided by a positive invariant set under the contact Hamiltonian flow $X_H$, or equivalently by the monotonicity along the solution semigroups $T^-_t,T^+_t $ , namely 

\begin{thmA}\label{theorem A}
Suppose that $\varphi(x) \in C(M,\R)$, the following statements are equivalent:
\begin{enumerate}
	\item[\rm(1)] $\varphi (x)$ is a  viscosity subsoltion of equation \eqref{eq:intro_HJs}, i.e.  $H(x,p,\varphi(x))\leqslant 0 $ for any $x\in M, p\in D^+ \varphi(x)$.
	\item[\rm(2)] $T^-_t \varphi(x) \geqslant \varphi(x) $ , for any $t\geqslant 0, x\in M$; 
	\item[\rm(3)] $T^+_t \varphi(x)\leqslant \varphi(x)$ for any $t\geqslant 0, x\in M$;
	\item[\rm(4)]  $\Phi^t_H(\Gamma_\varphi  ) \subset \Gamma_\varphi $ for any $t\geqslant 0$, where $\Gamma_\varphi:=\{(x,p,u),u\geqslant \varphi(x) \} $. 
	\end{enumerate}
\end{thmA}

In \cite{Fathibook}, a function $\varphi(x)\in C(M,\R)$ is called  a viscosity subsolution of $\eqref{eq:intro_HJs}$ is strict at $x_0\in M$  if there exists an open neighborhood $V_{x_0}$ of $x_0$, and $c_{x_0}<0$ such that $u|{V}_{x_0}$ is a viscosity subsolution of $H(x,d_xu)=c_{x_0} $ on $V_{x_0}$. We define that $\varphi(x)$ is a \textbf{strict viscosity subsolution} of $\eqref{eq:intro_HJs}$ on $M$ if  $\varphi(x)$ is  a viscosity subsolution of $\eqref{eq:intro_HJs}$ which is strict at each  $x\in M$. In Theorem B,  we give the necessary and sufficient conditions for the  strict viscosity subsolution .
\begin{thmB}\label{theorem B}
Suppose that $\varphi(x) \in C(M,\R)$, the following statements are  equivalent: 
\begin{enumerate}
	\item[\rm(1)] 
	$\varphi (x)$ is a strict viscosity subsoltion of equation \eqref{eq:intro_HJs} on $M$.
	% i.e. there exists $c>0$ such that 
	% $\varphi (x)$ is a  viscosity subsoltion of equation with $H+c$, where $c$ is a positive constant. 
	\item[\rm(2)]  $T_t^-\varphi(x) >\varphi(x) $ for any $t\geqslant 0, x\in M$ and  there exists $c>0$ such that
	 \begin{equation}\label{eq:T_t- +c}
	    \quad \lim_{t\to 0^+} \frac{1}{t}  \Big( T_t^- \varphi (x)- \varphi(x)\Big) \geqslant c, \quad \forall x\in M.
	 \end{equation}
	\item[\rm(3)] $T_t^+\varphi(x) < \varphi(x) $ for any $t\geqslant 0, x\in M$ and  there exists $c>0$ such that
	 \begin{equation}\label{eq:T_t+ +c}
	  \lim_{t\to 0^+} \frac{1}{t}  \Big( T_t^+ \varphi (x)- \varphi(x)\Big) \leqslant -c, \quad \forall x\in M.
	 \end{equation}
%	and
%	$$
%	 \lim_{t\to 0^+} \frac{1}{t} \Big(\pi_u(\Phi^t_H(\Gamma_\varphi ) )-\varphi(\pi_x(\Phi^t_H(\Gamma_\varphi ) ) )   \Big)\geqslant c>0.
%	$$
%	where $\pi_x: T^*M \times \mathbb{R} \rightarrow M$ and $\pi_u: T^*M \times \mathbb{R} \rightarrow \R$ are the canonical projection.
%	\item[\rm(5)] For any $x\in M$, $H(x,p,\varphi(x))+c\leqslant 0 $ for any $p\in D^+ \varphi(x)$.
\end{enumerate}
\end{thmB}
\begin{corC}
	$T_t^-\varphi(x) >\varphi(x) $ for any $t\geqslant 0, x\in M$ if and only if $\Phi^t_H(\Gamma_\varphi  ) \subset \mathring \Gamma_\varphi $ for any $t> 0$.
\end{corC}

The rest of this paper is organized as follows. We give some preliminaries and prove Theorem A in Section 2. The proof of Theorem B and Corollary C are given in Section 3. 
%In Appendix, we show the completeness of the flow.

\section{Proof of theorem A}
In this section, we first recall the definition of the viscosity solution of equation\eqref{eq:intro_HJs} and implicit action functions and some properties of them. Then we give the proof of theorem A.

\medskip
\subsection{Preliminaries} 
\begin{definition}[Viscosity solutions of equation \eqref{eq:intro_HJs}] 
Let $U$ be an open subset $U \subset M $.
	\begin{itemize}
		\item [(1)]A function $u:U\to \R $ is called a viscosity subsolution of equation \eqref{eq:intro_HJs}, if for every $C^1$ function $\phi:U\rightarrow \R$ and every point $x_0 \in U$ such that $u- \phi $ has a local maximum at $x_0$, we have 
		$
		H(x_0,d_x \phi(x_0),u(x_0) )\leqslant c;
		$
		\item[(2)]A function $u:U\to \R $ is called a viscosity supersolution of equation \eqref{eq:intro_HJs}, if for every $C^1$ function $\psi :U\rightarrow \R$ and every point $y_0 \in U$ such that $u- \psi $ has a local minimum at $y_0$, we have 
		$
		H(y_0,d_x \psi(y_0),u(y_0) )\geqslant c;
		$
		\item[(3)]A function $u:U\to \R $ is called a viscosity solution of equation \eqref{eq:intro_HJs}, if it is both a viscosity subsolution and a viscosity supersolution.
	\end{itemize}	
\end{definition}
We recall that for any $x\in M$ and continuous function $u$,  the closed convex sets
\begin{align*}
	D^- u(x)= \Big\{ p\in T^*_x M: \lim \inf_{y\to x} \frac{u(y)-u(x)-\langle p,y-x \rangle }{|y-x|}\geqslant 0 \Big\},\\
	D^+ u(x)= \Big\{ p\in T^*_x M: \lim \sup_{y\to x} \frac{u(y)-u(x)-\langle p,y-x \rangle }{|y-x|}\leqslant 0 \Big\}.
\end{align*}   
are called the  \textit{ subdifferential and superdifferential} of $u$ at $x$ , respectively, see \cite{BC,Ba,CS,F2} for more details on this notion and its relationship with viscosity solutions. For instance, it is proved in \cite[Prop 3.1.7]{CS} that  
$p\in D^+ u(x)$ if and only if $p= D \phi(x) $ for some $C^1$ function $\phi$ such that $u-\phi$ attains a local maximum at $x$. Thus, we can present the other notion of the  viscosity subsolution(supersolution) as 
$$
H(x,p,u) \leqslant 0(\ \geqslant 0 \ ), \quad \forall p\in D^+u(x) \ ( \ D^-u(x) \ ).
$$

 Let us recall two implicit action functions introduced in \cite{WWY2}\cite{WWY3} which can give the representation formulae for $T_t^-$ and $T_t^+$ by \cite[Proposition 4.1]{WWY3}:
\begin{equation}\label{semigroup and action}
T^-_t\varphi(x)=\inf_{y\in M}h_{y,\varphi(y)}(x,t),\quad T^+_t\varphi(x)=\sup_{y\in M}h^{y,\varphi(y)}(x,t),\quad (x,t)\in M\times(0,+\infty),
\end{equation}
where the continuous functions 
\begin{align*}
	h_{x_0,u_0}(x,t):M\times\R\times M\times(0,+\infty)\to\R, \qquad h^{x_0,u_0}(x,t):M\times\R\times M\times(0,+\infty)\to\R  \\ (x_0,u_0,x,t)\mapsto h_{x_0,u_0}(x,t) \qquad \qquad \qquad \qquad (x_0,u_0,x,t)\mapsto h^{x_0,u_0}(x,t)
\end{align*} were introduced in Proposition \ref{Minimality} and \ref{Maximality}, called forward and backward implicit action functions respectively. There are various properties of implicit action functions in  \cite{WWY2,WWY1,WWY3}.

%There are various properties of implicit action functions which we collected  in section \ref{Propositions of Implicit action functions}.

%From the aspects of the partial differential equation, in \cite{WWY1},they define the backward and forward solution semigroup $T^-_t,T^+_t$ , which can give a representation formula for the solution semigroup of the contact Hamilton evolutionary equation in \cite{WWY3}.

%\subsection{Propositions of Implicit action functions}\label{Propositions of Implicit action functions}

\begin{proposition}{\cite{WWY2}}\label{Minimality} 
Properties of backward implicit action function $h_{x_0,u_0}(x,t)$:
\begin{enumerate}
	\item[(1)] \text{(Minimality)} Given $x_0,x\in M $  and $u_0\in \R $ and $t>0$,let $S^{x,t}_{x_0,u_0}$ be the set of the solutions $(x(s),p(s),u(s) ) $ of \eqref{eq:ode} on $[0,t]$ with $x(0)=x_0,x(t)=x,u(0)=u_0$.Then 
\begin{equation}
\label{eq:Minimality}
 h_{x_0,u_0}(x,t):=\inf \{u(t):(x(s),p(s),u(s))\in S^{x,t}_{x_0,u_0}  \} ,
\end{equation} 
for any $ (x,t)\in M \times (0,+\infty )$.
	\item[(2)]\text{(Monotonicity)} Given $x_0 \in M, u_1< u_2 \in \R $, we have $h_{x_0,u_1}(x,t)< h_{x_0,u_2}(x,t)$
	for any $t>0, x\in M$.
	\item[(3)]\text{(Lipschitz continuity)}The function $(x_0,u_0,x,t)\mapsto h_{x_0,u_0}(x,t)$ is locally Lipschitz continuous on $M\times \R\times M\times (0,+\infty ) $.
	\item[(4)] \text{(Implicit variational)} For any given $x_0\in M $ and $u_0\in \R $, 
	%there exists a continuous function $h_{x_0,u_0}(x,t)$ defined on $M\times (0,+\infty) $ satisfying
	\begin{equation} \label{eq:Implicit variational}
	h_{x_0,u_0}(x,t)=u_0+\inf_{\substack{\gamma(t)=x\\ \gamma(0)=x_0 } }\int_0^t L(\gamma(\tau), \dot \gamma(\tau),h_{x_0,u_0}(\gamma(\tau) ,\tau )  )\ d\tau,
	\end{equation}
where the infimum is taken among the Lipschitz continuous curves $\gamma:[0,t]\rightarrow M $ and can be achieved. 
\end{enumerate}
\end{proposition}

 \begin{proposition}{\cite{WWY3}}\label{Maximality}
Properties of forward implicit action function $h^{x_0,u_0}(x,t)$:
\begin{enumerate}
	\item[\rm{(1)}] \text{(Maximality)} Given $x_0,x\in M $  and $u_0\in \R $ and $t>0$,let $S_{x,t}^{x_0,u_0}$ be the set of the solutions $(x(s),p(s),u(s) ) $ of \eqref{eq:ode} on $[0,t]$ with $x(0)=x,x(t)=x_0,u(t)=u_0$. Then 
\begin{equation}
\label{eq:Maximality}
 h^{x_0,u_0}(x,t):=\sup \{u(0):(x(s),p(s),u(s))\in S_{x,t}^{x_0,u_0}  \},
\end{equation} 
for any $ (x,t)\in M \times (0,+\infty )$.
   \item[\rm{(2)}]\text{(Monotonicity)} Given $x_0 \in M, u_1< u_2 \in \R $, we have
	$
	h^{x_0,u_1}(x,t)< h^{x_0,u_2}(x,t)
	$
	for any $t>0, x\in M$.
	\item[\rm{(3)}]\text{(Lipschitz continuity)} The function $(x_0,u_0,x,t)\mapsto h^{x_0,u_0}(x,t)$ is locally Lipschitz continuous on $M\times \R\times M\times (0,+\infty ) $.
	\item[\rm{(4)}]\text{(Implicit variational)}For any given $x_0\in M $ and $u_0\in \R $, 
	$$
h^{x_0,u_0}(x,t)=u_0-\inf_{\substack{\gamma(0)=x\\ \gamma(t)=x_0 } }\int_0^t L(\gamma(\tau), \dot \gamma(\tau),h^{x_0,u_0}(\gamma(\tau) ,t-\tau )  )\ d\tau,
$$ 
where the infimum is taken among the Lipschitz continuous curves $\gamma:[0,t]\rightarrow M $ and can be achieved.
\end{enumerate}
\end{proposition}
The relation between $h_{x_0,u_0}(x,t)$ and $h^{x_0,u_0}(x,t)$ was shown as follows:
\begin{proposition}{\rm{ \cite[Prop 3.5]{WWY3} }\text{(Equivalence)}}\label{Equivalence}
$h_{x_0,u_0}(x,t)=u \Leftrightarrow h^{x,u}(x_0,t)=u_0.$
\end{proposition}

\subsection{Proof of Theorem A}
Before turning to the proof of the main Theorem A, we will need one more preliminary.
\begin{lemma}\label{lem Phi}
	Given $\varphi \in C(M,\mathbb{R})$, we have
	$$\Phi^t_H(\Gamma_\varphi) \subset \Gamma_{T_t^-\varphi}, \quad t\geqslant 0. $$ 
\end{lemma}	

\begin{proof}
	For any $(x_1,p_1,u_1) \in \Phi^t_H(\Gamma_\varphi)$, we assume that:  $$(x_1,p_1,u_1)=\Phi^t_H(x_0,p_0,u_0), \quad u_0\geqslant \varphi(x_0)$$
	
%	Due to (1) of Proposition \ref{Minimality},
%	$$
%	h_{x_0,u_0}(x_1,t)= \inf \{u(t):(x(s),p(s),u(s))\in S^{x_1,t}_{x_0,u_0}  \}  
%	$$
	
	Thus, by \eqref{semigroup and action} and Proposition \ref{Minimality}, it have 
	\begin{align*}
	T_t^- \varphi(x_1)=&\,\inf_{y\in M} h_{y,\varphi(y)}(x_1,t)
	\leqslant h_{x_0,\varphi(x_0)}(x_1,t)
	\leqslant h_{x_0,u_0}(x_1,t)\\ =&\, \inf \{u(t):(x(s),u(s),p(s))\in S^{x_1,t}_{x_0,u_0}\}
	\leqslant u_1.
	\end{align*}
	%where the first line is according to  \eqref{semigroup and action} and the second line is according to (2) of Proposition \ref{Minimality}.
	 It follows that $\Phi^t_H(\Gamma_\varphi) \subset \Gamma_{T_t^-\varphi} $ for any $t\geqslant 0$.
	
%	For any $x_1\in M$, $u_1=T_{t}^-\varphi(x_1)$, we would like to choose some $p_1\in T_{x_1}^*M$ such that  $(x_1,p_1,u_1) \in \Phi^t_H(\Gamma_\varphi)$.
%	Choose $(x_0,u_0) \in M \times \mathbb{R}$, such that
%	$$  \left\{
%		\begin{aligned}
%		&h_{x_0,\varphi(x_0)}(x,t)=\inf_{y\in M}h_{y,\varphi(y)}(x,t) \\
%		&u_0=\varphi(x_0) 
%		\end{aligned}
%		\right.
%	$$
%	
%	Let $u_1=\inf \{u(t):(x(s),u(s),p(s))\in S^{x_1,t}_{x_0,u_0}\}$ with the corresponding $p_1=p(t)$, $p_0=p(0)$. Then 
%		\begin{align*}
%	T_t^- \varphi(x_1)=&\,\inf_{y\in M}h_{y,\varphi(y)}(x_1,t)=h_{x_0,\varphi(x_0)}(x_1,t)=h_{x_0,u_0}(x_1,t)\\
%	=&\, \inf \{u(t):(x(s),u(s),p(s))\in S^{x_1,t}_{x_0,u_0}\}= u_1.
%	\end{align*}
%	and $(x_1,p_1,u_1)=\Phi^t_H(x_0,p_0,u_0) \in \Phi^t_H(\Gamma_{\varphi } ) $, which implies that 
%	$\textbf{graph} \ T^-_t \varphi\subset \pi \circ \Phi_H^t(\Gamma_\varphi)$ for any $t\geqslant 0$.
\end{proof}

%\begin{remark}
%	In the description of Lemma \ref{lem Phi}, $\varphi$ is not necessary to be a viscosity subsolution.	Besides, for the lower semi-continuous function $\varphi(x)$, $T^-_t\varphi(x)$ is also well-defined, and satisfies Lemma \ref{lem Phi}.
%	
%	In a few words, Lemma \ref{lem Phi} shows that the evolution by the backward solution semigroup is along the minimal characteristic curve.
%\end{remark}

We get back to the proof of Theorem A and assume that $\varphi(x)$ is Lipschitz continuous at first.

\noindent{\it Proof of Theorem A:} 
$(1) \Rightarrow (2)$: 
	We claim that: 
	If  $\varphi(x)$ is a viscosity subsoltion of \eqref{eq:intro_HJs},then for any piecewise $C^1$ curve $\gamma:[a,b]\rightarrow M$, 
	$$
	\varphi(\gamma(b))-\varphi(\gamma(a))\leqslant \int_a^b L(\gamma(t),\dot \gamma(t),\varphi(\gamma(t))) dt,   
	$$
	
	 If $\varphi(x)$ is differentiable at $x_0$, then $H(x_0,\partial_x\varphi(x_0),\varphi(x_0))\leqslant 0$. Following the method mentioned in \cite[Prop 4.2.3]{Fathibook}, we can always choose a sequence of piecewise $C^1$ curves $\gamma_n:[a,b]\to M$, such that $\varphi$ is differentiable on $\gamma_n(t)$ for almost every $t\in [a,b]$, $\gamma_n(a)=\gamma(a)$,$\gamma_n(b)=\gamma(b)$ and $\gamma_n$  converges in the $C^1$ topology to $\gamma $, then we obtain 
	 \begin{align*}
	 	&\, \varphi(\gamma(b))-\varphi(\gamma(a))
	 	=  \liminf_{n\to \infty} \int_a^b \frac{d \varphi(\gamma_n(t) ) }{dt} dt = \liminf_{n\to \infty} \int_a^b \langle \partial_x \varphi(\gamma_n(t)), \dot \gamma_n(t) \rangle  dt \\
	 	\leqslant &\, \liminf_{n\to \infty} \int_a^b L(\gamma_n(t),\dot \gamma_n(t),\varphi(\gamma_n(t)) )+H(\gamma_n(t),\dot \gamma_n(t),\varphi(\gamma_n(t)) )dt\\
	 	\leqslant &\, \liminf_{n\to \infty} \int_a^b L(\gamma_n(t),\dot \gamma_n(t),\varphi(\gamma_n(t)) )dt \\
	 	=&\, \int_a^b L(\gamma(t),\dot \gamma(t),\varphi(\gamma(t)) )dt,
	 \end{align*}
	 which completes our claim. 
	 
	 Then we want to show $T_t^- \varphi \geqslant \varphi $ for any $t\geqslant 0 $. By contradiction, we assume that there exists  $ x_0 \in M$ and $t>0$ such that $T^-_t \varphi(x_0) <\varphi(x_0) $.
	 
	 Let $\gamma:[0,t]\to M$ be a minimizer of  \eqref{eq:T^-_t} with $\gamma(t)=x_0 $, i.e.
	 $$
	 T_t^- \varphi(x_0)= \varphi(\gamma(0))+\int_0^t L(\gamma(s),\dot\gamma(s),T_s^-\varphi (\gamma(s)))\ ds.
	 $$
	 
	 Let $G(s)=\varphi(\gamma(s))-T_s^- \varphi(\gamma(s))$, since $G(0)=0$ and $G(t)>0$, then there exists $0\leqslant t_0<t$ such that $G(t_0)=0$ and $G(s)>0 , \forall s\in (t_0,t] $ .
	 \begin{align*}
	 	&\,G(s)= \varphi(\gamma(s))-T_s^-\varphi(\gamma(s))= \varphi(\gamma(s))-\varphi(\gamma(t_0))-\int_{t_0}^s L(\gamma(\tau),\dot\gamma(\tau),T_\tau^-\varphi (\gamma(\tau)))\ d\tau \\
	 	\leqslant &\, \int_{t_0}^s L(\gamma(\tau),\dot\gamma(\tau),\varphi (\gamma(\tau)))-L(\gamma(\tau),\dot\gamma(\tau),T_\tau^-\varphi (\gamma(\tau)))\ d \tau     
	 	\leqslant  \lambda \int_{t_0}^s G(\tau) \ d \tau,
	 \end{align*}
	 where $\lambda$ is the Lipschitz constant of $L$ with respect to $u$. From Gronwall inequality ,it follows that $G(s)\leqslant 0 $ for any $s \in [t_0,t]$. It derives a contradiction, which follows that $T_t^-\varphi \geqslant \varphi $ for any $t\geqslant 0$. \\

\medskip
$(2) \Rightarrow (1)$: 
 Since $T_t^-\varphi \geqslant \varphi $ for any $t\geqslant 0$, by 
 %the  definition of $T^-_t \varphi $ and 
 \eqref{semigroup and action}, we have
  $$
  T_t^- \varphi (y)=\inf_{x\in M} h_{x,\varphi(x)}(y,t) \geqslant \varphi(y) .
 $$
 
 For any fixed $x\in M$ and any $ C^1 $ curve $\xi:[0,t]\to M$ with $\xi(0)=x$ and $\xi(t)=y$, due to \eqref{eq:Implicit variational}, we have
 \begin{equation}\label{eq:A2-1.1}
\begin{split}
	\varphi(y)\leqslant &\, h_{x, \varphi(x)}(y,t) 
	=\varphi(x) +\inf_{\substack{\gamma(t)=y \\ \gamma(0)=x } }\int_0^t L(\gamma(\tau), \dot \gamma(\tau) ,h_{x, \varphi(x)}(\gamma(\tau) ,\tau ) )\ d\tau, \\
	\leqslant &\, \varphi(x) +\int_0^t L(\xi(\tau), \dot \xi(\tau),h_{x, \varphi(x)}(\xi (\tau) ,\tau  )  )\ d\tau.
\end{split}
 \end{equation}
For any differentiable point $x\in M$ of $\varphi $,
\begin{align*}
	\lim_{t \rightarrow 0^+}\frac{1}{t}[\varphi (\xi(t))-\varphi (\xi(0)) ]\leqslant &\, \lim_{t\rightarrow 0^+}\frac{1}{t}\int_0^t L(\xi(s),\dot\xi(s),h_{x,\varphi (x)}(\xi(s),s)) \ ds, 
%	\\
%    = &\,  L(\xi(0),\dot\xi(0),h_{x,\varphi (x)}(\xi(0),0))\\
%    = &\,  L(x,\dot \xi(0),\varphi(x)).
\end{align*}
which leads to
$$
\langle \partial_x\varphi(x),\dot\xi(0)\rangle \leqslant L(\xi(0),\dot\xi(0),h_{x,\varphi (x)}(\xi(0),0)) =L(x,\dot \xi(0),\varphi(x)).
$$

By taking $\dot\xi(0)=\partial_p H(x,\varphi (x),\partial_x \varphi(x))$, we get
\[
L(x,\dot \xi(0),\varphi(x))+H(x,\partial_x \varphi(x),\varphi(x))=\langle \partial_x\varphi(x),\dot\xi(0)\rangle,
\]
which implies $H(x, \partial_x \varphi (x) ,\varphi (x))\leqslant 0.$

Since $\varphi$ is Lipschitz on $M$, $\varphi$ is differentiable almost everywhere, then $\varphi(x):M\rightarrow\R$ is an almost everywhere subsolution. As a result, it has to be a viscosity subsolution of \eqref{eq:intro_HJs},  where the equivalence between almost everywhere subsolutions and viscosity subsolutions was proved in bunch of references \cite{BC,Ba,Si}.\\

\medskip	
$(2)\Rightarrow (4) $: 
From Lemma \ref{lem Phi}, we have $ \Phi^t_H(\Gamma_\varphi) \subset \Gamma_{T_t^-\varphi}$ for any $t\geqslant 0$. Due to (2), it get $\Gamma_{T_t^-\varphi} \subset \Gamma_\varphi $ for any $t\geqslant 0$.  
	Hence, $\Phi^t_H(\Gamma_\varphi ) \subset \Gamma_\varphi $ for any $t\geqslant 0$. \\

\medskip	
$(4)\Rightarrow (2)  $:	 
 For any given $x\in M$ and $t\geqslant 0$, by Lemma \ref{lem Phi}, there exists $p\in T_x M$ such that
 $$ (x,p,u)\in \Phi^t_H(\Gamma_\varphi), \text{   and }
  u=T^-_t \varphi(x). $$
Since $(x,p,u)\in \Phi^t_H(\Gamma_\varphi)\subset \Gamma_\varphi$, it follows that $ T^-_t \varphi(x)=u\geqslant \varphi(x)$ for any $t\geqslant 0 , x\in M$.\\

\medskip	
$(2)\Rightarrow	 (3) $: Given $t>0$,due to  \eqref{semigroup and action}, it gets that 
    $$
    T^-_t \varphi(y)=\inf_{x \in M} h_{x,\varphi(x)}(y,t)\geqslant \varphi(y) \quad \forall y \in M,
    $$
    which implies that
    $ h_{x,\varphi(x)}(y,t)\geqslant \varphi(y)$ for any $ x,y \in M$.

    Let $u=h_{x,\varphi(x)}(y,t)$, and by Proposition \ref{Equivalence}, if follows that $  h^{y,u}(x,t)=\varphi(x)$. 
    From (2) of Proposition \ref{Maximality}, we have
   $
     h^{y,\varphi(y)}(x,t)\leqslant\varphi(x)
    % \Rightarrow	\quad \sup_{y \in M} h^{y,\varphi(y)}(x,t)\leqslant \varphi(x), \quad \forall x \in M 
    $ for any $x,y \in M$. Thus, by  \eqref{semigroup and action}, it follows that $T^+_t \varphi(x)=\sup_{y \in M} h^{y,\varphi(y)}(x,t)  \leqslant \varphi(x)$ for any $x \in M $.\\
     
$(3)\Rightarrow	 (2) $:   
The proof is similar to the process above, and we omit it here. 
\qed

\begin{remark}\label{rmk:sub Lip}
Recall that 
	%In fact, in Theorem A, it is not necessary to assume that $\varphi(x)$ is Lipschitz continuous.  
	 any viscosity subsolution of \eqref{eq:intro_HJs} has to be Lipschitz continuous by \cite[Lemma 2.2]{ishii}, and it is easy to verify that $\varphi(x)$ is Lipschitz if it satisfies (2). As the statement above, it is sufficient to assume that $\varphi(x)$ is continuous when we adapt it into Theorem A.
	
%	If $ T^-_t \varphi\geqslant \varphi $ for any $t\geqslant 0$, from \eqref{eq:A2-1.1}, we have 
%	$$
%	\varphi(y)\leqslant \varphi(x) +\int_0^t L(\xi(\tau), \dot \xi(\tau),h_{x, \varphi(x)}(\xi (\tau) ,\tau  )  )\ d\tau,\qquad \forall x,y\in M.
%	$$
%	
%	For each $x,y\in M$, let $\xi:[0,d(x,y)]\to M$ be a geodesic length $d(x,y)$, parameterized by arc length and connecting $x$ to $y$. Since $M$ is compact, $\varphi(x)$ is continuous,(4) of Proposition \ref{Minimality},let
%	$$
%	A_1:=\max_{\substack{ x,y\in M \\ s\in [0,\rm{diam}  M  ] } }|h_{x,\varphi(x)}(y,s)|  \quad A_2:= \sup \{L(x,u,\dot x)|x\in M,|u|\leqslant A_1,\|\dot x\|=1 \}
%	$$  
%	
%	Due to  $\|\dot \xi(\tau) \|=1$ for each $\tau \in [0,d(x,y)]$, we have 
%	$$
%	\varphi(y)-\varphi(x)\leqslant \int_0^{d(x,y)} L(\xi(\tau), \dot \xi(\tau),h_{x, \varphi(x)}(\xi (\tau) ,\tau)) \ d \tau \leqslant A_2 \ d(x,y),
%	$$
%	where $ $
%	which implies that  $\varphi(x)$ is Lipschitz continuous by exchanging the roles of $x$ and $y$.
\end{remark}

\section{Proof of Theorem B and Corollary C}
\noindent{\it Proof of Theorem B:}  
$(1)\Rightarrow (2)$:
%For $c>c_1$ ,denote that $H_c(x,p,u):=H(x,p,u)-c$ and $\widetilde T_t^-$ determined by $H-c_1$, then by \eqref{eq:Minimality} and (5) of Proposition \ref{eq:Minimality},for any $t>0$ and $x\in M$,
%	$$
%	T_t^-\varphi(x)=\inf_{y\in M}h^{L+c}_{y,\varphi(y)}(x,t) > \inf_{y\in M}h^{L+c_1}_{y,\varphi(y)}(x,t) =\widetilde T_t^-\varphi(x) \geqslant \varphi(x).
%	$$
%	It  is similar for $T_t^+$.
	By Remark \ref{rmk:sub Lip}, $\varphi(x)\in C(M)$ is a Lipschitz strict viscosity subsoltion of equation \eqref{eq:intro_HJs} and then for any $x_0 \in M$, there exists an open neighborhood $V_{x_0}$ of $x_0$, and $c_{x_0}<0$ such that $u|{V}_{x_0}$ is a viscosity subsolution of $H(x,d_xu)=c_{x_0} $ on $V_{x_0}$.
	
	 Due to $M$ is compact,  every covering of $M$ contains a finite subcollection covering, then there  exists $c>0$ such that  $\varphi(x)\in C(M)$ is a Lipschitz strict viscosity subsoltion of \eqref{eq:intro_HJs} with $H+c$. 
%	It implies that  there  exists $c>0$ such that 
%	$$
%	H(x,p,\varphi(x))\leqslant -c \quad \forall x\in M, p\in D^+ \varphi(x).
%	$$ 

	Similar with Theorem A, we can choose a sequence of piecewise $C^1$ curves $\gamma_n:[a,b]\to M$, such that $\varphi$ is differentiable on $\gamma_n(t)$ on $ [a,b]$, and $\gamma_n$  converges in the $C^1$ topology to $\gamma $, then we obtain 
	\begin{equation}\label{eq:B.1}
	\begin{split}
	&\,  \varphi(\gamma(b))-\varphi(\gamma(a))= \lim_{n\to +\infty}  \varphi(\gamma_n(b))-\varphi(\gamma_n(a)) =\lim_{n\to +\infty} \int_a^b \langle d_x \varphi(\gamma_n(t)), \dot \gamma_n(t) \rangle  dt \\ 
	\leqslant &\, \lim_{n\to +\infty} \int_a^b L(\gamma_n(t),\dot \gamma_n(t),\varphi(\gamma_n(t)) )+H(\gamma_n(t),\dot \gamma_n(t),\varphi(\gamma_n(t)) )  dt\\
\leqslant &\, \lim_{n\to +\infty}  \int_a^b  L(\gamma_n(t),\dot \gamma_n(t),\varphi(\gamma_n(t)) )-c \  dt \\
=&\, \int_a^b  L(\gamma(t),\dot \gamma(t),\varphi(\gamma(t)) )-c \  dt.
	\end{split}
	\end{equation}	
	 By Theorem A, $\varphi(x)$ is a viscosity subsoltion of equation \eqref{eq:intro_HJs} and $T_t^- \varphi \geqslant \varphi$ for any $t\geqslant 0$.  We claim that 	 
	 \begin{equation}
	     T_t^- \varphi(x) \geqslant \varphi(x)+ c\cdot \frac{1-e^{-\lambda t}}{\lambda } \quad \forall x\in M ,t\in [0,+\infty) .
	 \end{equation}
	 By contradiction, we assume that  there exists  $ x\in M$ and $t>0$ such that $T^-_t \varphi(x)< \varphi(x)+c\cdot \frac{1-e^{-\lambda t}}{\lambda }  $.Due to \eqref{eq:T^-_t}, there exists $\gamma:[0,t]\to M$ be a minimizer of \eqref{eq:T^-_t}  with $\gamma(t)=x $, i.e.
	 \begin{equation}\label{eq:B.2}
	 	 T_t^- \varphi(x)= \varphi(\gamma(0))+\int_{0}^t L(\gamma(\tau),\dot\gamma(\tau ),T_\tau ^-\varphi (\gamma(\tau)))\ d\tau.
	 \end{equation}
	  Let $G(s)=\varphi(\gamma(s))+c\cdot \frac{1-e^{-\lambda s}}{\lambda }  -T_s^- \varphi(\gamma(s)) $,	 since $G(0)=0$ and $G(t)>0$, then there exists $0\leqslant t_0<t$ such that $G(t_0)=0$ and $G(s)>0 $ for any $ s\in (t_0,t] $.
According to \eqref{eq:B.1} and \eqref{eq:B.2}, it shows that
	 \begin{align*}
	 	G(s)=&\,G(s)-G(t_0) \\
	 	=&\, \varphi(\gamma(s))+c\cdot \frac{1-e^{-\lambda s}}{\lambda } -T_s^-\varphi(\gamma(s))-  \varphi(\gamma(t_0))-c\cdot \frac{1-e^{-\lambda t_0}}{\lambda }+ T_{t_0}^-\varphi(\gamma(t_0))  \\
	 	=&\, \varphi(\gamma(s))-\varphi(\gamma(t_0))-\int_{t_0}^s L(\gamma(\tau),\dot\gamma(\tau),T_\tau^-\varphi (\gamma(\tau)))\ d\tau+c\cdot \frac{e^{-\lambda t_0} -e^{-\lambda s}}{\lambda }  \\
	 	\leqslant &\, \int_{t_0}^s \Big[ L(\gamma(\tau),\dot\gamma(\tau),\varphi (\gamma(\tau)))-L(\gamma(\tau),\dot\gamma(\tau),T_\tau^-\varphi (\gamma(\tau)))\Big] \ d \tau  -c(s-t_0) +c\cdot \frac{e^{-\lambda t_0}-e^{-\lambda s}}{\lambda }    \\
	 	\leqslant &\, \lambda \int_{t_0}^s \Big|T_\tau^-\varphi (\gamma(\tau))-\varphi(\gamma(\tau) )  \Big| \ d \tau -c(s-t_0)+c\cdot \frac{e^{-\lambda t_0}-e^{-\lambda s}}{\lambda }  \\
	 	< &\, \lambda \int_{t_0}^s c\cdot \frac{1-e^{-\lambda \tau }}{\lambda }  d \tau -c(s-t_0) +c\cdot \frac{e^{-\lambda t_0}-e^{-\lambda s}}{\lambda }=0 .
	 \end{align*}
	 where $\lambda$ is the Lipschitz constant of $L$ with respect to $u$. It derives a contradiction, which follows that 
	 $$
	 T_t^- \varphi(x) \geqslant \varphi(x)+ c \cdot \frac{1-e^{-\lambda t}}{\lambda }> \varphi(x) \quad \forall x\in M ,t\in (0,+\infty) .
	 $$ 
	Moreover, we have
	 $$
	 \lim_{t\to 0^+} \frac{1}{t}  \Big( T_t^- \varphi (x)- \varphi(x)\Big) \geqslant \lim_{t\to 0^+}  c\cdot \frac{1-e^{-\lambda t }}{\lambda t}=c>0, \quad \forall x\in M.
	 $$

\medskip

$(2)\Rightarrow (1)$: 
Due to \eqref{semigroup and action}, 
\begin{equation}\label{3.1}
	\lim_{t\to 0^+} \frac{1}{t}  \Big( h_{y,\varphi(y)}(x,t)  - \varphi(x)\Big) \geqslant \lim_{t\to 0^+} \frac{1}{t} \Big( T_t^- \varphi (x)- \varphi(x)\Big) \geqslant c>0, \quad \forall x\in M.
\end{equation}
 For any fixed $x\in M$ and any $ C^1 $ curve $\xi:[0,t]\to M$ with $\xi(0)=y$ and $\xi(t)=x$, due to \eqref{eq:Implicit variational}, we have
\begin{align*}
	  h_{y, \varphi(y)}(x,t) 
	=&\,\varphi(y) +\inf_{\substack{\gamma(t)=x \\ \gamma(0)=y } }\int_0^t L(\gamma(\tau), \dot \gamma(\tau) ,h_{y, \varphi(y)}(\gamma(\tau) ,\tau ) ) \ d\tau, \\
	\leqslant &\, \varphi(y) +\int_0^t L(\xi(\tau), \dot \xi(\tau),h_{x, \varphi(x)}(\xi (\tau) ,\tau  )  )\ d\tau.
\end{align*}

For any differentiable point $x\in M$ of $\varphi $,putting  it into \eqref{3.1}, we obtian that
\begin{align*}
	c+\lim_{t \rightarrow 0^+}\frac{1}{t}[\varphi (\xi(t))-\varphi (\xi(0)) ]\leqslant &\, \lim_{t\rightarrow 0^+}\frac{1}{t}\int_0^t L(\xi(s),\dot\xi(s),h_{x,\varphi (x)}(\xi(s),s))  \ ds, 
%	\\
%    = &\,  L(\xi(0),\dot\xi(0),h_{x,\varphi (x)}(\xi(0),0))\\
%    = &\,  L(x,\dot \xi(0),\varphi(x)).
\end{align*}
which leads to
$$
c+\langle \partial_x\varphi(x),\dot\xi(0)\rangle \leqslant L(\xi(0),\dot\xi(0),h_{x,\varphi (x)}(\xi(0),0)) =L(x,\dot \xi(0),\varphi(x)).
$$

By taking $\dot\xi(0)=\partial_p H(x,\varphi (x),\partial_x \varphi(x))$, we get
\[
L(x,\dot \xi(0),\varphi(x))+H(x,\partial_x \varphi(x),\varphi(x))=\langle \partial_x\varphi(x),\dot\xi(0)\rangle,
\]
which implies
$
H(x, \partial_x \varphi (x) ,\varphi (x))\leqslant -c .$

\medskip
$(2)\Leftrightarrow (3)$: It is similar with $(2)\Leftrightarrow (3)$ of Theorem A.
\qed

%	\begin{corollary}
%	$T_t^-\varphi(x) >\varphi(x) $ for any $t\geqslant 0, x\in M$ if and only if $\Phi^t_H(\Gamma_\varphi  ) \subset \mathring \Gamma_\varphi $ for any $t\geqslant 0$.
%\end{corollary}

\vspace{20pt}

\noindent{\it Proof of Corollary C:} It is a strict version of  Theorem A. On one hand, due to Lemma \ref{lem Phi},  $\Phi_H^t(\Gamma_\varphi )\subset \Gamma_{T_t^-\varphi} $. Thus, by $T_t^- \varphi>\varphi$, we have	 $\Gamma_{T_t^-\varphi} \subset \mathring{\Gamma}_\varphi $. 
On the other hand, for any given $x\in M$ and $t> 0$, by (\ref{semigroup and action}) and (\ref{eq:Minimality}), there exists $p\in T_x M$ such that
 $ (x,p,u)\in \Phi^t_H(\Gamma_\varphi)$ and $ u=T^-_t \varphi(x)$. Since $(x,p,u)\in \Phi^t_H(\Gamma_\varphi)\subset \mathring \Gamma_\varphi$, it follows that
$  T^-_t \varphi(x)=u> \varphi(x). $
Hence, $T_{t}^-\varphi(x)> \varphi(x)$ for any $t> 0$ and $x\in M$.
\qed

\medskip

\appendix
\section{Completeness of the flow}
%\begin{lem}\label{lem:Complete flow}
%	Suppose that $H$ satisfies (H1)-(H3) and (SH1) ,then the contact vector field $\Phi_H^t$ generates a complete flow on $T^*M\times \R$, which means the flow exists for $t\in (-\infty ,+\infty)$.  
%\end{lem}
\begin{lemma}\label{lem:complete flow}
	Suppose that $H$ satisfies (H1)-(H3) and there exists a continuous function $A(h):\R\to \R^+$ such that 
	\begin{equation}\label{eq:sufficient}
	\Big|p\cdot \frac{\partial H}{\partial p} (x,p,u)\Big|\leqslant A\big(H(x,p,u)\big)(1+|u|), \quad \forall (x,p,u)\in T^*M\times \R. 
	\end{equation}
Then condition (SH1) holds, i.e. the contact vector field $\Phi_H^t$ generates a complete flow on $T^*M \times \R$.
\end{lemma}
\noindent{\it Proof of Lemma \ref{lem:complete flow}:}
	Denote that $H(t):=H(\Phi_H^t(x,p,u))=H(x(t),p(t),u(t))$ for any $t\in \R$. By \eqref{eq:ode}, one can compute that
	\begin{equation*}
	\begin{aligned}
	&\frac{d}{dt} H(\Phi_H^t(x,p,u)) = \Big[ \frac{\partial H}{\partial x} \cdot \dot{x} +\frac{\partial H}{\partial p} \cdot \dot{p}+\frac{\partial H}{\partial u} \cdot \dot{u} \Big](x(t),p(t),u(t))  \\
	=&\, \Big[\frac{\partial H}{\partial x} \cdot  \frac{\partial H}{\partial p} - \frac{\partial H}{\partial p} \cdot (\frac{\partial H}{\partial x} +\frac{\partial H}{\partial u} \cdot p ) + \frac{\partial H}{\partial u} \cdot (\frac{\partial H}{\partial p} \cdot \dot{p}-H ) \Big](x(t),p(t),u(t))   \\
	=&\,-\frac{\partial H}{\partial u}(x(t),p(t),u(t))  \cdot H(x(t),p(t),u(t)) .
	\end{aligned}
	\end{equation*}
	which together with (H3), implies that $|H(x(t),p(t),u(t))| \leqslant e^{\lambda |t|}  |H(x(0),p(0),u(0))| $ . \\
	From (SH1), 	$|p\cdot \frac{\partial H}{\partial p} |\leqslant A(H)(1+|u|)$, then by \eqref{eq:ode} it obtain 
	\begin{align*}
	|\dot u|= \Big|p\cdot \frac{\partial H}{\partial p} -H\Big|\leqslant A(H)(1+|u|)+|H| 
	\leqslant  A(h_0)|u|+ A(h_0)+ e^{\lambda |t|}|H(0)|. 
	\end{align*}
	where $A(h_0)=\sup_{|h|\leqslant e^{\lambda |t|}H(0)}A(h)$. It shows that $|u(s)|$ is bounded  on $[-t,t]$  for any $t \geqslant 0$.
	%   Since $|H(t)|$ is bounded by $t$, $|\dot{u}|\leqslant\alpha(t)(|u|+1)$, where $\alpha(t)>0$ is a continuous function for $t\in [0,+\infty)$. 
	%   Therefore, $u(t)$ is bounded by $t$.\\
	If there does not exist a flow $\Phi_H^t(x,p,u)$ for $t\in (-\infty,+\infty)$, it means that the solution $(x(t),p(t),u(t))$ can not be extended further for some finite $t_0$. It implies that $|p(t_0)|$ would blow up to infinite for some finite $t_0$, which leads to the boundless of $|H(t_0)|$ because of (H2). It contradicts that  $|H(t)|$ is bounded by finite $t$. Therefore, the contact vector field $\Phi_H^t$ generates a complete flow.

\vskip 1cm

\section*{Acknowledgements} 

Jun Yan is supported by NSFC Grant No.  11631006, 11790273.All the authors are grateful to  Prof. Lin Wang for helpful suggestions.

\medskip

\bibliographystyle{abbrv}
%\bibliography{mybib}
%\begin{thebibliography}{99}
%\bibitem{WWY1} Wang K, Wang L \& Yan J, CMP
%\bibitem{WWY2} Wang K, Wang L \& Yan J, NL
%\bibitem{WWY3} Wang K, Wang L \& Yan J, JMPA
%\bibitem{VSII} Wang K, Wang L \& Yan J,VSII
%\bibitem{Fathibook} Fathi,book
%\bibitem{JY} Jin,Yan 
%\end{thebibliography}
%\bibliographystyle{plainurl}
\bibliography{sub.bib}

\end{document}